\documentclass{amsproc}

\usepackage{amssymb}
\usepackage{graphicx}
\usepackage{psfrag,color}

\newtheorem{theorem}{Theorem}[section]
\newtheorem{lemma}[theorem]{Lemma}
\newtheorem{proposition}[theorem]{Proposition}
\newtheorem{corollary}[theorem]{Corollary}

\theoremstyle{definition}
\newtheorem{definition}[theorem]{Definition}

\theoremstyle{remark}
\newtheorem{remark}[theorem]{Remark}

\numberwithin{equation}{section}

\newcommand{\ZZ}{{\mathbb  Z}}
\newcommand{\RR}{{\mathbb  R}}
\newcommand{\CC}{{\mathbb  C}}
\newcommand{\NN}{{\mathbb  N}}
\newcommand{\QQ}{{\mathbb  Q}}

\newcommand{\OO}{{\mathcal O}}
\newcommand{\C}{{\mathcal C}}

\begin{document}

\title{Abelian and non-abelian numbers via 3D Origami}

\author[J.I.~Royo Prieto]{Jos\'{e} Ignacio Royo Prieto}
\address{Department of Applied Mathematics, University of the Basque Country UPV/EHU, Alameda de Urquijo s/n 48013 Bilbao, Spain.}
\email{joseignacio.royo@ehu.es}
\thanks{The first author is partially supported by the UPV/EHU grant EHU12/05.}

\author[E.~Tramuns]{Eul\`alia Tramuns}
\address{Departament de Matem\`atica Aplicada IV, Universitat Polit\`ecnica de Catalunya, Barcelona Tech.
C/ Jordi Girona 1-3, M\`odul C3, Campus Nord, 08034 Barcelona, Spain}
\curraddr{}
\email{etramuns@ma4.upc.edu}
\thanks{The second author is partially supported by MTM2011-28800-C02-01 from Spanish MEC}

\subjclass[2010]{11R32, 12F05, 51M15, 51M20.}

\date{\today}

\begin{abstract}
In this work we introduce new folding axioms involving easy 3D manoeuvres with the aim to push forward the arithmetic limits of the Huzita-Justin axioms. Those 3D axioms involve the use of a flat surface and the rigidity property of convex polyhedra. Using those folding moves, we show that we can construct all Abelian numbers, and numbers whose Galois group is not solvable.
\end{abstract}

\maketitle

\section{Introduction}

\subsection{The scope of Origami folding axioms}

A classical result of Number Theory  (see, for example, \cite[Corollary 10.1.7]{cox}) establishes that the set $\C$ of points constructed using  ruler and compass (RC, for short) with initial set $\{0,1\}$ is the smallest subfield of $\mathbb{C}$ which is closed under the operation of taking square roots. Thus, constructions that involve solving  cubic equations, such as the trisection of an arbitrary angle or finding the side length of a cube whose volume is the double of another given cube, are not possible in general using RC. More than two thousand years before those constructions were proven to be impossible by Pierre Wantzel, Archimedes found a verging construction (i.e. {\em neusis} construction) with a marked ruler which made it possible to trisect any given angle. So, the usage of new tools (formally, new axioms) allowed mathematicians to go beyond $\C$.

It is well known that the description of the origami folding operations described by the Huzita-Justin axioms (HJAs, for short, see section \ref{sec:hja}) leads to the set of points constructible with origami, which is the smallest subfield $\OO$ of the complex plane $\CC$ which is closed under the operations of taking square roots, cubic roots and complex conjugation (see \cite[Theorem 5.2]{alperin}). The elements of $\OO$ are called {\em origami numbers}. Thus, we have $\C\subset\OO$ as fields. Notice that this inclusion is a strict one: origami numbers like $\cos(\pi/9)$, $\sqrt[3]{2}$ and the seventh root of unity $\zeta_7=e^{2\pi i/7}$ do not belong to $\C$ (and thus, the trisection of an angle of 60 degrees, the duplication of a cube of volume 1 and the construction of a regular heptagon are not possible with RC). Moreover using the HJAs one can solve any cubic equation (see Lemma \ref{lem:cubica}), thus allowing all the geometric constructions described above. Nevertheless, there are algebraic numbers like the 11-th root of unity $\zeta_{11}=e^{2\pi i/11}$ which are are not constructible using origami, as modelized by the HJAs.

We are interested in extending the scope of the referred origami axioms, that is, in adding new axioms to the HJAs in order to get fields larger than $\OO$, and studying the arithmetic properties of those new numbers.

\subsection{Beyond the Huzita-Justin axioms: $n$-fold axioms}
\label{subsec:nfolds}

The description of the folding moves given by  the HJAs just covers a small part of the manoeuvres usually performed by origamists, and it is natural to consider extending the set of folding axioms.
In \cite{alperinlang} it is proved that the HJAs form a {\em complete set of 1-fold axioms} in the sense that they describe all possible combinations of alignments (among points and lines) which may be achieved by one single fold. As a consequence, no more folding axioms are to be found, if we stick to axioms where just one fold is performed. In the same article new axioms requiring the performance of $n$ simultaneous folds are introduced, considering alignments which involve the new lines determined by the creases of those $n$ folds. For example, one of the 2-fold axioms consists in trisecting a given angle (notice that the alignments required involve the input lines defining the angle and the output ones that trisect it). The payoff of considering those moves as axioms is enormous from the arithmetic point of view: using Lill's method to evaluate polynomials (see \cite{lill}), it is shown how to find the real roots of any polynomial of degree $n$ and coefficients in $\ZZ$ using $(n-2)$-fold axioms. In turn every real algebraic number can be reached.

\subsection{New 3D folding axioms}

In this work we propose new folding axioms that, added to the HJAs, allow us to construct new numbers.
The new ingredient we propose consists in incorporating 3D origami moves, that is, folds that do not leave the paper flat after being executed. We want the folding procedures behind our axioms to satisfy the following properties:

\begin{enumerate}

\item the folds involved should be fully referenced;
\item the manoeuvres should represent relatively easy and common folding;
\item the alignment skills needed to perform the moves should be similar to those needed in the HJAs.

\end{enumerate}

The third condition rules out the $n$-fold axioms of subsection \ref{subsec:nfolds} for $n\ge2$. Thus, we shall stick to moves easier to perform, perhaps renouncing to the big arithmetic achievement those $n$-fold axioms provide. Our first additional axioms are based on the fact that, starting with a regular polygon of $n+1$ sides, an easy folding sequence using 3D moves (see Fig. \ref{fig:heptagon}) yields a regular polygon of $n$ sides. We call those new axioms the Regular Polygon Axioms (RPAs for short). Using the RPAs and the HJAs and starting with $\{0,1\}$, we can get regular polygons with any number of sides (in particular, we can construct a regular 11-gon, which was not possible with the HJAs), and as we show in section \ref{sec:regular}, this leads us to all the algebraic numbers whose Galois group is Abelian. Our second axioms generalize the RPAs by starting with a cyclic polygon (i.e., a polygon inscribed in a circle). We call them the Cyclic Polygon Axioms (CPAs, for short), and they take us further on, so that we can get numbers whose Galois group is non-solvable (in particular, those numbers cannot be expressed using radicals).

This article is organized as follows: in section \ref{sec:hja} we briefly recall the HJAs and the characterization of the origami numbers. In section \ref{sec:pyramids} we describe the origami moves that lead us to the new axioms and discuss their axiomatization. In section \ref{sec:regular} we introduce the RPAs and describe the number fields we get using them. In section \ref{sec:cyclic} we establish the geometric validity of our constructions and deal with the CPAs. Finally, in section \ref{sec:conclusions} we discuss the arithmetic limits of polyhedral constructions, and pose some open questions for further research.

\section{Origami numbers}
\label{sec:hja}

In this section we briefly recall the Huzita-Justin axioms (HJAs, for short) and their arithmetic scope. We refer to the introduction of  \cite{alperinlang} for a good account of their discovery, history and references. Those axioms are the following:

\begin{description}

\item[O1] we can fold the line connecting two given points;
\item[O2] we can fold the mediatrix of two given points;
\item[O3] we can fold the bisector of the angle defined by two given lines;
\item[O4] we can fold, through a given point, a perpendicular to a given line;
\item[O5] we can fold a line passing through a given point that places another given point onto a given line;
\item[O6] we can fold a line that places simultaneously two given points onto two given lines, respectively;
\item[O7] we can fold a line perpendicular to a given line that places a given point onto another given line.

\end{description}
In the literature of constructions, the phrases ``we can fold'' and  ``given'' stand for ``we can construct'' and ``constructible''.
For technical reasons, we shall also consider the following intersection axiom:
\begin{description}
\item[LineIntersection (LI for short)] the intersection point of two given non-parallel lines is constructible.
\end{description}
In the literature, sometimes this axiom is only implicitly assumed, but notice that
it is needed in order to get points. As a convention, in the sequel, we shall consider that the HJAs is the set of axioms O1-O7, along with LI.

\begin{definition}
We say that $\alpha\in\CC$ is an {\em origami number} if there is a finite sequence of the HJAs that starts with $\{0,1\}$ and ends in $\alpha$. We denote the set of all origami numbers as $\OO$.
\end{definition}
Given $0$ and $1$, axioms O1 and O6 easily construct anything the others do. The latter yields the simultaneous tangents to the parabolas given by the points and lines of the input, and involves solving of a cubic equation. In fact, we have the following result, which proof we reproduce, essentially, following \cite[section 5.1]{alperin}:
\begin{lemma}\label{lem:cubica}
Consider a number field $K$ as initial set. Then, the application of the HJAs yields all the solutions of any cubic equation with coefficients in $K$.
\end{lemma}
\begin{proof}
Field operations allow us to perform a change of variable and reduce ourselves to the case $x^3+a x+b=0$. Consider the parabolas $(y-a/2)^2=2bx$ and  $y=x^2/2$,
whose directrices and foci are constructible using field operations involving $a,b$. Axiom O6 constructs a line that is tangent to  both. A straightforward computation (see \cite[section 5.1]{alperin}) shows that the slope $m$ of that line satisfies $m^3+a m+ b=0$. Notice that, given an oblique line $\ell$, its slope is realized as the vertical side of a straight triangle with horizontal side of length 1 and hypothenuse parallel to $\ell$. So, using the HJAs, we can construct a real root $m$ of the cubic. Using $m$ and just field operations, we can factor the cubic (e.g. using Ruffini's scheme), reducing ourselves to finding the roots of a quadratic equation, which is solvable using compass and straight edge constructions, and thus, with the HJAs.
\end{proof}

The following result summarizes some characterizations of $\OO$ (see \cite[10.3]{cox} for a proof of part (\ref{torre}) just adding O6 to the usual axioms of ruler and compass):

 \begin{theorem}
 The set $\OO\subseteq\mathbb{C}$ of origami numbers is the smallest subfield of $\mathbb{C}$ which is closed under the operations of taking square and cubic roots and complex conjugation (\cite{alperin}). Moreover, we have that $\alpha\in\OO$ is equivalent to each of the following:
\begin{enumerate}

\item $\alpha$ is constructible by marked ruler (\cite{martin});

\item $\alpha$ is constructible by intersecting conics (\cite{alperin});

\item $\alpha$ lies in a 2-3 tower $\mathbb{Q}=F_0\subset\dots\subset F_n$ (\cite{videla}).
\label{torre}

\end{enumerate}
\end{theorem}

\begin{figure}
\psfrag{1}[][]{1}
\psfrag{2}[][]{2}
\psfrag{3}[][]{3}
\psfrag{4}[][]{4}
\psfrag{5}[][]{5}
\psfrag{6}[][]{6}
\psfrag{7}[][]{7}
\psfrag{8}[][]{8}
\includegraphics[scale=0.5]{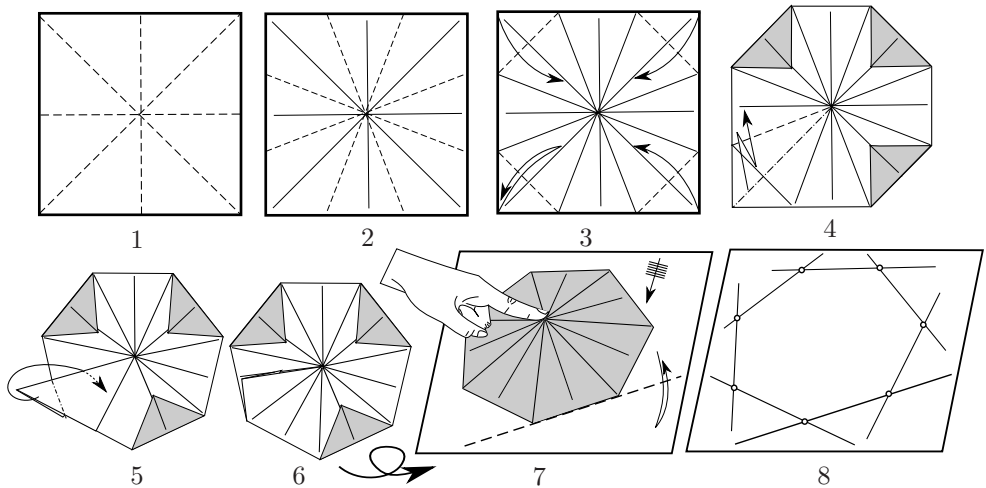}
\caption{usage of 3D to crease the sides of a regular heptagon}
\label{fig:heptagon}
\end{figure}

\begin{remark}
As a consequence of (\ref{torre}), every origami number belongs to a solvable extension of $\QQ$. Not all origami numbers have Abelian Galois group. Namely, the only real root $m$ of $x^3 - 2x - 2$, which is irreducible over $\QQ$ by Eisenstein's rule, is an origami number due to (\ref{torre}). By the following standard exercise in Galois Theory (\cite[p.139]{cox}), its Galois group is the non-Abelian symmetric group $S_3$.
\end{remark}

\begin{lemma}
\label{lema:symmetric}
Let $f\in\ZZ[x]$ be irreducible over $\QQ$ with $\mathop{\rm deg}(f)=p$ prime and exactly $p-2$ real roots. Then, its Galois group over $\QQ$ is the symmetric group $S_p$.
\end{lemma}

\section{Origami folds behind the new 3D axioms}
\label{sec:pyramids}

In this section we present some easy 3D origami constructions which motivate the introduction of our new axioms.

Consider the folding sequence of Fig. \ref{fig:heptagon}. Steps 1-3 use the HJAs. In steps 4-6, no new line is constructed, but a pleat prevents the model from remaining flat. We get a flexible triangulated surface whose dihedral angles are not fixed. Now, if we put our finger in the apex and push gently against a flat surface, we get a right regular pyramid, which yields a physical origami realization of a regular heptagon in the plane, say a paper underneath. Folding that paper upwards along the sides of the pyramid we retrieve the result as a set of creases. In Fig. \ref{fig:tato8} we have a more artistic folding sequence rendering the same flexible surface, based on a traditional octagonal tato we learnt about thanks to the initial folds of \cite{flowertower}.

It is remarkable that, in both models, the side faces of the pyramid remain flat despite all the extra creases. This is a consequence of the classic Rigidity Theorem:

\begin{theorem}[Legendre-Cauchy]
\label{th:cauchy}
Any two convex polyhedra with the same graph and congruent corresponding faces are congruent.
\end{theorem}

\begin{figure}
\psfrag{1}[][]{1}
\psfrag{2}[][]{2}
\psfrag{3}[][]{3}
\psfrag{4}[][]{4}
\psfrag{5}[][]{5}
\psfrag{stuff}[][]{Push to stuff inside}
\psfrag{refold}[][]{Refold tato}
\psfrag{open}[][]{Open tato}
\includegraphics[scale=0.35]{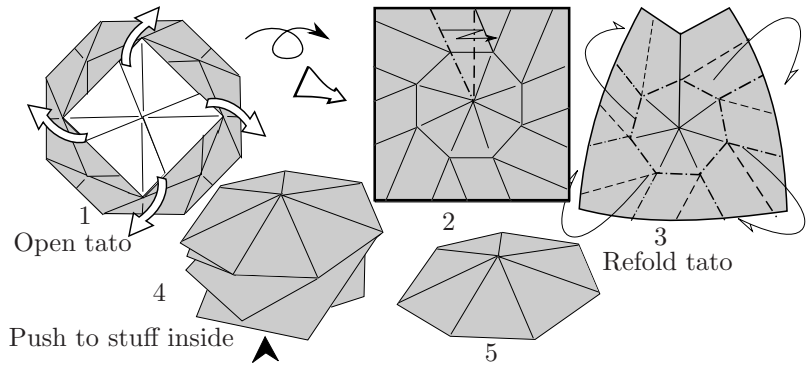}
\caption{alternative folding sequence using an octagonal tato}
\label{fig:tato8}
\end{figure}

 Notice that, for the sake of clarity, we have used two separate pieces of paper to show the construction, but it is not hard to get everything with just one sheet of paper (e.g. \cite{royo}, or the pyramidal eyes of \cite[step 65]{engel}).

The axiomatization of this construction can be done in many ways. One could be to develop a theory of constructible points in the three-dimensional space, describing every possible move with some axiom.
Instead, as our interest is focused on the arithmetic aspects of the construction, we have chosen to keep working in the plane, and define the axioms using an abstract input-output scheme.
In this case, the input of the axiom is the set of vertices of a regular $(n+1)$-gon, and its output is the set of lines defining the sides of a regular $n$-gon (see Fig. \ref{fig:axiomR}).

In the second axiom (see Fig. \ref{fig:axiomC}),
we drop regularity, just asking the starting polygon to be {\em cyclic}: its vertices belong to a circle.
Given a cyclic polygon with sides of lengths $a_1,\dots,a_{n+k}$, steps 4-8 of Fig. \ref{fig:heptagon} retrieve a cyclic polygon of side lengths $a_1,\dots,a_n$, where $k=1,2$ is the number of sides sacrificed by the pleat in step 4 (it is easy to pleat and lock more than one triangle).

We will prove the geometric validity of these constructions in Proposition \ref{prop:cuentaca}.

\section{Regular Polygon Axioms}
\label{sec:regular}

In the sequel, for the sake of brevity, we shall adopt the following convention.
\begin{description}
\item[Convention] Unless otherwise explicitly stated, we shall suppose that the vertices and side lengths of a polygon are listed in counterclockwise order.
\end{description}
Based on the first construction of the previous section, we present our first 3D axioms, the {\em Regular Polygon Axioms} (RPAs, for short). For $n\ge 3$, we have:
\begin{description}
\item[$\text{RPA}_n$] Given the vertices $A_1,\dots,A_{n+1}$ of a regular $(n+1)$-gon, we can fold the line containing any side of the regular $n$-gon of vertices $B_1,\dots,B_n$, determined by $B_1=A_1$ and $B_2=A_2$.
\end{description}

\begin{figure}
\psfrag{A1}[][]{$A_1$}
\psfrag{A2}[][]{$A_2$}
\psfrag{A3}[][]{$A_3$}
\psfrag{A4}[][]{$A_4$}
\psfrag{A5}[][]{$A_5$}
\psfrag{A6}[][]{$A_6$}
\psfrag{A7}[][]{$A_7$}
\psfrag{A8}[][]{$A_8$}
\psfrag{B3}[][]{$B_3$}
\psfrag{B4}[][]{$B_4$}
\psfrag{B5}[][]{$B_5$}
\psfrag{B6}[][]{$B_6$}
\psfrag{B7}[][]{$B_7$}
\includegraphics[scale=0.35]{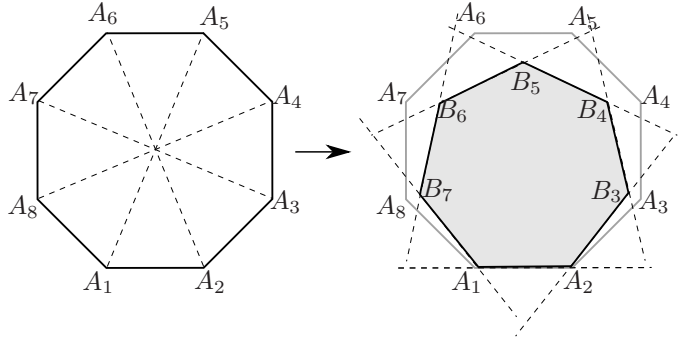}
\caption{Input and output of $\text{RPA}_7$}
\label{fig:axiomR}
\end{figure}

Now we describe the arithmetic consequences of the application of the RPAs.
\begin{lemma}
\label{lem:zeta}
Assume the notation in the definition of the axiom $\text{RPA}_n$, and let $K$ be a number field containing $A_1,\dots,A_{n+1}$. Then, we have
$$
K(B_1,\dots,B_n)=K(\zeta_n)
$$
\end{lemma}
\begin{proof}
Let $C$ be
the circumcenter of the resulting polygon.
We have:
\begin{equation}
\label{eq:generandoBi}
B_i-C=(\zeta_n)^{i-1}(B_1-C),\quad\forall i=2,\dots,n
\end{equation}
which gives $K(B_1,\dots,B_n)\subset K(B_1,B_2,\zeta_n,C)=K(\zeta_n)$, because $B_i=A_i\in K$ for $i=1,2$ and $C=(A_2-A_1\zeta_n)/(1-\zeta_n)$. For the converse, as
 $C=(B_1+\dots+B_n)/n$, we have $C\in K(B_1,\dots,B_n)$. Now, $\zeta_n=(B_2-C)/(B_1-C)$ yields $K(\zeta_n)=K(\zeta_n,C)\subset K(B_1,\dots,B_n)$ and we are done.
\end{proof}

\begin{remark}\label{rem:newline}
In the conditions of lemma \ref{lem:zeta}, the coefficients and slope $m$ of the Cartesian equation of the lines containing the sides of $P$ are easily constructed from $K(\zeta_n)$ as described in the proof of Lemma \ref{lem:cubica}. Thus, they belong to some 2-extension of $K(\zeta_n)$.
\end{remark}

\begin{definition}
We say that $\alpha\in\CC$ is a {\em regular origami number} if there is a finite sequence of the HJAs and the RPAs that starts with $\{0,1\}$ and ends in $\alpha$. We denote the set of all regular origami numbers as $\OO_{RP}$.
\end{definition}

\begin{lemma}
\label{lem:todaszetas}
The field $\OO_{RP}$ contains $\zeta_n$, for every $n\in\NN$.
\end{lemma}
\begin{proof}
Fix $n\ge 3$, and take $k>\log_2(n+1)$. Divide the square into $2^{k+1}$ equal angles using bisections, like in step 2 of Fig. \ref{fig:heptagon}. Folding perpendiculars, we obtain a regular polygon of $2^k$ sides whose vertices belong to the final field $F$ of a 2-3 tower, by \eqref{torre}. We get
a new regular polygon of $m=2^{k}-1>n$ sides from $\text{RPA}_m$. Its vertices belong to $F(\zeta_{m})$ by lemma \ref{lem:zeta}. Applying  the RPAs successively, we get regular polygons of  $m,m-1,\dots,n$ sides, thus proving that $F(\zeta_{n})\subset\OO_{RP}$.
\end{proof}

Denote by $\QQ_{Ab}$ the maximal Abelian extension of $\QQ$. Now, we recall:

\begin{theorem}[Kronecker-Weber]\label{thm:kronecker}
Every finite Abelian extension $K\subset\CC$ of $\mathbb{Q}$ is a subfield of a cyclotomic field (i.e., $K\subset\QQ(\zeta_k)$, for some $k\in\NN$).
\end{theorem}

As a consequence of the celebrated theorem \ref{thm:kronecker}, $\QQ_{Ab}$ is the smaller extension of $\QQ$ that contains all the roots of unity. So, by
 Lemma \ref{lem:todaszetas}, we have that $\OO_{RP}$ is an extension of $\QQ_{Ab}$.
Now we present the main result of this section:

\begin{theorem}
$\OO_{RP}$ is the smallest subfield of $\CC$ containing $\QQ_{Ab}$ and which is closed under the operations of taking square roots, cubic roots and complex conjugation. Moreover, $\alpha\in\OO_{RP}$ if and only if there exists a tower
\begin{equation}
\label{eq:tower}
\mathbb{Q}=F_0\subset\dots\subset F_n
\end{equation}
such that $\alpha\in F_n$ and, for $i=0,\dots,n-1$,
either
$2\le[F_{i+1}:F_{i}]\le 3$
or
$F_{i+1}=F_i(\zeta_k)$ for some $k\in\NN$
(we call such a tower a 2-3-c tower; c for cyclotomic).
\end{theorem}
\begin{proof}
For the first part of the statement we adapt the argument of \cite[Theorem 4.2]{alperin}.
By Lemmas \ref{lem:cubica} and \ref{lem:todaszetas}, the minimal subfield satisfying the conditions of the statement must be a subfield of $\OO_{RP}$. The converse follows from the fact that the application of the HJAs just involve field operations or solving quadratic or cubic equations, and from Lemma \ref{lem:zeta}.

For the second part of the statement, we proceed as in \cite[Theorem 10.3.4]{cox}.
If $\alpha\in F_n$ belongs to tower \eqref{eq:tower} , then Lemma \ref{lem:cubica} permits us to solve each cubic or quadratic equation involved, and the RPAs allow us to adjoin any needed root of unity.
Conversely, given a number $\alpha\in\OO_{RP}$, we can proceed by induction on the number of axioms applied to construct $\alpha$. Notice that in order to apply the proof of \cite[Theorem 10.3.4]{cox} we just need to show that the coefficients and slopes of the lines involved in the application of the last used axiom belong to a 2-3-c tower,and this follows from Remark \ref{rem:newline} in the case of the RPAs and from the arguments given in \cite[Theorem 10.3.4]{cox} in the case of the other axioms. By juxtaposition with the 2-3-c tower given by the induction argument, we get a 2-3-c tower. \end{proof}

\begin{corollary}
\label{cor}
Every regular origami number has solvable Galois group.
\end{corollary}
\begin{proof}
Every $\alpha\in\OO_{RP}$ belongs to a 2-3-c tower \eqref{eq:tower}. As every extension $F(\zeta_k)/F$ is radical (thus, solvable) and the juxtaposition of solvable extensions is solvable, we have the result.
\end{proof}

\section{Cyclic Polygon Axioms}
\label{sec:cyclic}

In this section, we generalize the RPAs by starting with a cyclic polygon, not necessarily regular. We first recall some known facts about cyclic polygons.

\subsection{Cyclic polygons}
\label{subsec:cyclic}

Recall that a {\em cyclic} polygon is a polygon whose vertices belong to a circle. Let's see that for positive numbers $a_1,\dots,a_n$ satisfying
\begin{equation}\label{eq:condition}
2\left(\max\limits_{1\le i\le n} a_i\right)\le \sum\limits_{1\le i\le n}a_i,
\end{equation}
there exists a unique convex cyclic polygon with sides in that order. Following \cite[Section 1]{pak}, take a circle with large enough radius $r$, so that placing vertices on it at distances $a_1,\dots,a_n$ in that order, we get an open polygonal. Shrinking $r$ until the polygonal is closed, we get a convex polygon due to condition \eqref{eq:condition}.
The radius of the circumscribed circle is called the {\em circumradius} of the cyclic polygon.

The circumradius, the area and the lengths of the diagonals of cyclic polygons have been object of active research in the last two decades (see \cite{pak} for a survey). They, or their squares, are roots of polynomials with coefficients in $\QQ(a_1^2,\dots,a_n^2)$. Those polynomials are difficult to compute, but explicit formulas can be found for cyclic pentagons, hexagons and heptagons, of degrees 7, 14 and 38, respectively (see  \cite{robbins},\cite{sabitov}, \cite{varfolomeev1} and  \cite{moritsugu}). General formulas are, nowadays, computationally intractable as the degree and the number of summands quickly becomes gargantuan.

\subsection{Validity of the origami constructions}
In this subsection, we prove the geometric legitimacy of the origami constructions of section \ref{sec:pyramids}.

\begin{proposition}
\label{prop:cuentaca}
Consider a flexible polyhedral surface $S$ imbedded in $\RR^3$, formed by a closed circuit of isosceles triangles, glued along their equal sides around some point $O$. Assume that the sum of the angles at $O$ is less than $2\pi$. Then there exists only one polygon that, glued to $S$ by the boundary, forms a pyramid.
\end{proposition}

\begin{proof}
Denote the triangles by $T_1,\dots,T_n$ in counterclockwise order. For $i=1,\dots,n$, denote by $R$ the length of the equal sides of $T_i$, $a_i$ the remaining one, and $\theta_i$ its opposite angle. Note that $\sin{(\theta_i/2)}=a_i/(2R)$. Without loss of generality, we suppose $a_1=\max\{a_1,\dots,a_n\}$. We claim that condition \eqref{eq:condition} holds. By reductio ad absurdum, suppose $a_1>a_2+\dots+a_n$. Then,
\begin{equation}
\label{eq:cuentaca}
\sin\frac{\theta_1}{2}=\frac{a_1}{2R} >\frac{a_2+\dots+a_n}{2R}
       =\sin\frac{\theta_2}{2} +\dots+\sin\frac{\theta_n}{2}
       \ge\sin\frac{\theta_2+\dots\theta_n}{2},
\end{equation}
where we have repeatedly used $\sin\alpha_1+\sin\alpha_2 \ge \sin(\alpha_1+\alpha_2)$, which follows immediately from the sum of sines formula. By hypothesis, both $\theta_1/2$ and $(\theta_2+\dots+\theta_n)/2$ belong to $(0,\pi)$. So, by \eqref{eq:cuentaca}, it follows that $\theta_1>\theta_2+\dots+\theta_n$, a contradiction.

So, $a_1,\dots,a_n$ are the side lengths of a cyclic convex polygon $P$.
Denote by $r$ its circumradius. Notice that $r<R$ by construction. Consider in $\RR^3$ the sphere of radius $R$ centered at $A=(0,0,\sqrt{R^2-r^2})$, whose intersection with the plane $z=0$ is a circle of radius $r$. Placing $P$ in that circle and joining its vertices with $A$, we get a pyramid that satisfies the statement (see Fig. \ref{fig:bipiramide}). To prove the uniqueness of $P$, consider any other such pyramid. Then, the vertices of its base belong to the intersection of its plane with the sphere of radius $R$ centered at its apex. Hence, the base is a cyclic convex polygon and thus, unique.
\end{proof}

\begin{figure}
\begingroup%
  \makeatletter%
  \providecommand\color[2][]{%
    \errmessage{(Inkscape) Color is used for the text in Inkscape, but the package 'color.sty' is not loaded}%
    \renewcommand\color[2][]{}%
  }%
  \providecommand\transparent[1]{%
    \errmessage{(Inkscape) Transparency is used (non-zero) for the text in Inkscape, but the package 'transparent.sty' is not loaded}%
    \renewcommand\transparent[1]{}%
  }%
  \providecommand\rotatebox[2]{#2}%
  \ifx\svgwidth\undefined%
  %
  %
    \setlength{\unitlength}{333.25bp}%
    \ifx\svgscale\undefined%
      \relax%
    \else%
      \setlength{\unitlength}{\unitlength * \real{\svgscale}}%
    \fi%
  \else%
    \setlength{\unitlength}{\svgwidth}%
  \fi%
  \global\let\svgwidth\undefined%
  \global\let\svgscale\undefined%
  \makeatother%
  \begin{picture}(1,0.39194687)%
    \put(0,0){\includegraphics[width=\unitlength]{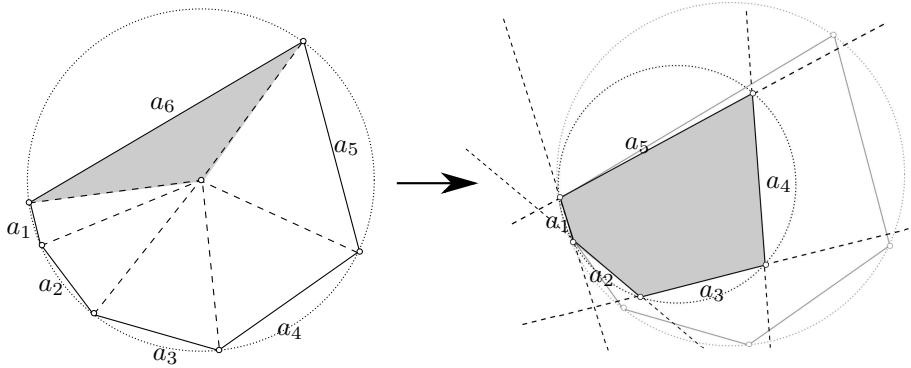}}%
    \put(-0.02086605,0.13237333){\color[rgb]{0,0,0}\makebox(0,0)[lb]{$a_1$}}%
    \put(0.01515277,0.06783962){\color[rgb]{0,0,0}\makebox(0,0)[lb]{$a_2$}}%
    \put(0.14471005,-0.01030429){\color[rgb]{0,0,0}\makebox(0,0)[lb]{$a_3$}}%
    \put(0.28529398,0.01782388){\color[rgb]{0,0,0}\makebox(0,0)[lb]{$a_4$}}%
    \put(0.34932222,0.22592745){\color[rgb]{0,0,0}\makebox(0,0)[lb]{$a_5$}}%
    \put(0.14170848,0.27444238){\color[rgb]{0,0,0}\makebox(0,0)[lb]{$a_6$}}%
    \put(0.58845106,0.13929289){\color[rgb]{0,0,0}\makebox(0,0)[lb]{$a_1$}}%
    \put(0.63830901,0.07821899){\color[rgb]{0,0,0}\makebox(0,0)[lb]{$a_2$}}%
    \put(0.76285871,0.06142719){\color[rgb]{0,0,0}\makebox(0,0)[lb]{$a_3$}}%
    \put(0.83810123,0.18428823){\color[rgb]{0,0,0}\makebox(0,0)[lb]{$a_4$}}%
    \put(0.67877588,0.22900789){\color[rgb]{0,0,0}\makebox(0,0)[lb]{$a_5$}}%
  \end{picture}%
\endgroup%
\caption{Input and output of $\text{CPA}_{5,1}$. The used lengths are  $(a_1,a_2,a_3,a_4,a_5)=(1,2,3,4,5)$, as in the proof of Theorem \ref{th:masalla}.}
\label{fig:axiomC}
\end{figure}
\begin{figure}
\psfrag{R}[][]{$R$}
\psfrag{r}[][]{$r$}
\includegraphics[scale=0.4]{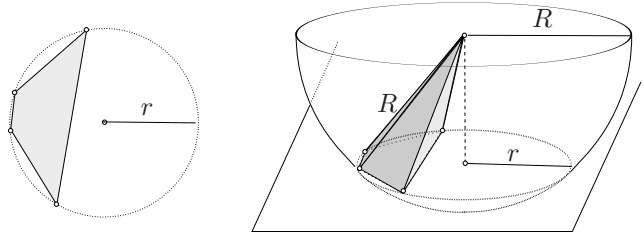}
\caption{Construction of a cyclic non-central polygon with a pyramid.}
\label{fig:bipiramide}
\end{figure}

\subsection{Cyclic Origami Numbers}
We denote by $Cy(a_1,\dots,a_n;A_1,A_2)$ the cyclic polygon of side lengths $a_1,\dots,a_n$ such that $A_1$ and $A_2$ are the vertices of the first side. We shall say that a cyclic polygon is {\em central} if its circumcenter belongs to its interior (this is needed to apply the folding sequence described in section \ref{sec:pyramids}). We present the {\em Cyclic Polygon Axioms} (CPAs, for short) for $n\ge3$ and $k=1,2$:

\begin{description}
\item[$\text{CPA}_{n,k}$] Given the vertices $A_1,\dots,A_{n+k}$ of the central cyclic polygon\newline $P_1=Cy(a_1,\dots,a_{n+k};A_1,A_2)$, we can fold the line containing any side of \newline $P_2=Cy(a_1,\dots,a_n;A_1,A_2)$, provided (1) the circumradius of $P_1$ is greater than that of $P_2$; and (2) $a_1,\dots,a_n$ satisfy condition \eqref{eq:condition}.
\end{description}

\begin{remark}
Condition (1) is needed for the origami construction (see proof of Proposition \ref{prop:cuentaca}). Condition (2) is needed for $P_2$ to exist. $P_2$ may not be central.
\end{remark}

Now, we show that the CPAs and the HJAs can construct any cyclic polygon.
\begin{lemma}
\label{lem:construcyclic}
Given two points $A_1,A_2\in\CC$ and lengths $a_1=|A_2-A_1|,a_2,\dots,a_n$ satisfying condition \eqref{eq:condition}, the vertices of $P=Cy(a_1,\dots,a_n;A_1,A_2)$ can be constructed using the HJAs and the CPAs.
\end{lemma}

\begin{proof}
Take $R\in\QQ$ greater than the circumradius of $P$. Consider the circle $C$ of radius $R$ passing through $A_1,A_2$. Name $O$ its center, and add vertices  $A_j$ in $C$ satisfying $|A_{j+1}-A_{j}|=a_{j}$ for $j=2,\dots,n$. If the angle $\widehat{A_1OA_{n+1}}$ is greater than $\pi$, then $Cy(a_1,\dots,a_n,|A_{n}-A_{n+1}|;A_1,A_2)$ is central, and $\text{CPA}_{n,1}$ gives $P$. If $\widehat{A_1OA_{n+1}}\le\pi$, let $A_2'$ be the antipodal point of $A_2$. Then, we have that $Cy(a_1,\dots,a_n,|A'_2-A_{n+1}|,|A_1-A_2'|;A_1,A_2)$ is central and $\text{CPA}_{n,2}$ gives $P$.
\end{proof}

\begin{remark}
Let $P$ be a cyclic polygon of vertices $A_1,\dots,A_n$, circumcenter $O$ and circumradius $r$ satisfying $\widehat{A_1OA_{n}}>\pi/2$. Then, it is easy to check that we can take $R$ slightly greater than $r$, so that $\text{CPA}_{n,1}$ gives $P$ with the method of Lemma \ref{lem:construcyclic}. Thus, $\text{CPA}_{n,2}$ is only strictly needed when $\widehat{A_1OA_{n}}\le\pi/2$. \end{remark}

\begin{definition}
We say that $\alpha\in\CC$ is a {\em cyclic origami number} if there is a finite sequence of the HJAs and the CPAs that starts with $\{0,1\}$ and ends in $\alpha$. We denote the set of all cyclic origami numbers as $\OO_{CP}$.
\end{definition}

Our last result shows that using the CPAs we can go beyond $\OO_{RP}$.

\begin{theorem}
\label{th:masalla}
The inclusion $\OO_{RP}\subset\OO_{CP}$ is a strict one.
\end{theorem}

\begin{proof}
The inclusion is clear because  both $\text{RPA}_n$ and $\text{CPA}_{n,1}$
yield the same result when applied to a regular $n$-gon. We now show that the inclusion is a strict one.
Consider $Cy(1,2,3,4,5;0,1)$, and write $d=|A_3-A_1|$. Notice that $d\in\OO_{CP}$. Using \cite[formula (1)]{varfolomeev} we obtain the minimal polynomial of $d$:
$$
P_d(x)= 4 x^7+ 51 x^6 + 160 x^5 - 246 x^4- 1836 x^3- 1785 x^2+ 1800 x +2160.
$$
Computer algebra (e.g. {\sc Mathematica$^{TM}$}) shows that $P_d$ is irreducible over $\QQ$. As
$P_d$ has exactly 5 real roots, by Lemma \ref{lema:symmetric} its Galois group over $\QQ$ is isomorphic to $S_7$, which is not solvable. By Corollary \ref{cor}, we have $d\notin\OO_{RP}$.
\end{proof}

\begin{remark}
In fact, in \cite[Theorem 1]{varfolomeev} it is shown that the Galois group of $P_d$ for any cyclic pentagon over the field generated by its side lengths is $S_7$.
\end{remark}

\section{Conclusions and further questions}
\label{sec:conclusions}

Our new axioms just describe a very small part of what can be determined in a plane using 3D origami.
In one hand, the constructions behind our axioms can be generalized by starting with any convex polygon $A_1,\dots,A_{n+1}$ and a chosen center $O$ inside it, so that $|A_1-O|=|A_{n+1}-O|$. If we sacrifice $\widehat{A_1OA_{n+1}}$ as in step 4 of Fig.~\ref{fig:heptagon}, we get a new $n$-gon. We could also
consider more than one node inside the starting polygon. Instead of a pyramid, we would get a different polyhedron.

It follows from  \cite[Theorem 1]{sabitov} that all the diagonals of a polyhedron $P$ are roots of a polynomial whose coefficients are rational numerical functions of the side lengths and of the diagonals of the faces of $P$. This implies that lengths obtained by constructing polyhedra with side lengths in some number field $K$ will be algebraic over $K$. So we cannot expect any {\em polyhedral origami number} to be transcendental.
It is left as a challenge a full characterization of $\mathcal{O}_{CP}$ in terms of field towers. It would also be nice to find explicit elegant folding sequences (not using compass) leading to the construction of cyclic polyhedra of given lengths.

\bibliographystyle{amsalpha}

\end{document}